\documentclass[hidelinks, 11pt]{article} 
\usepackage{amsthm}
\usepackage{amsfonts}
\usepackage{amssymb}
\usepackage[mathscr]{euscript}
\usepackage{amsmath} 
\usepackage[english, activeacute]{babel} 
\usepackage{graphicx} 
\usepackage[dvipsnames]{xcolor} 
\usepackage[utf8]{inputenc}
\usepackage{hyperref} 
\usepackage{captdef} 
\usepackage{caption}
\usepackage[toc,page]{appendix}

\newtheorem{prop}{Proposition}
\newtheorem{lem}{Lemma}
\newtheorem*{defi}{Definition}
\newtheorem{thm}{Theorem}

\theoremstyle{remark}

\begin{document}

\title{More about areas and centers of Poncelet polygons}
\author{Ana C. Chavez-Caliz}

\maketitle

\section{Introduction}

Inspired by the work of R. Schwartz and S. Tabachnikov exposed in \cite{centers-poncelet} about the loci of different centers of mass of Poncelet polygons, we study the locus of another center of mass: the Circumcenter of Mass (studied in detail in \cite{CCM}). \\

Let $P$ be an oriented $n$-gon $P=(V_1, \ldots , V_n)$ where each $V_i \in \mathbb{R}^2$. Let $A(P)$ denote the algebraic area of $P$ (which is defined by formula (\ref{area})). Given an appropriate\footnote{Turn to [14] for more specifics about the conditions of the triangulation.} triangulation  of $P$,  $P= \bigcup T_i$, let  $C_i$ and $A_i$ be the circumcenter and area of $T_i$ respectively.\\

Assume $A(P)$ is non-zero. The \emph{Circumcenter of Mass of $P$}, $CCM(P)$ is defined as the weighted sum of the circumcenters:
\begin{equation}
CCM(P)= \sum_{i} \frac{A_i}{A(P)}C_i.
\label{CCM}
\end{equation}

The Circumcenter of Mass does not depend on the triangulation and it can be described in terms of the vertices $V_i=(x_i, y_i)$ by the formula:

\begin{equation}
\begin{split}
CCM(P)= \frac{1}{4A(P)} & \Biggl( \sum_{i=1}^{n} y_i (x_{i-1}^2 + y_{i-1}^2 - x_{i+1}^2 - y_{i+1}^2), \\ 
						& \sum_{i=1}^{n} -x_i (x_{i-1}^2 + y_{i-1}^2 - x_{i+1}^2 - y_{i+1}^2) \Biggr).
\end{split}
\end{equation}

For more details and properties of the Circumcenter of Mass, consult \cite{CCM}. \\

Another center discussed in this paper is $CM_2$. If $P$ is a polygon, $CM_2(P)$ is the center of mass of $P$, considering $P$ as a ``homogeneous lamina''. Similarly to $CCM$, if $A(P)\neq 0$, $P=\bigcup T_i$ is a triangulation, and $G_i$, $A_i$ denote the centroid and area of each $T_i$ respectively, then $CM_2(P)$ is defined as the weighted sum of the centroids:

\begin{equation}
CM_2(P)= \sum_{i} \frac{A_i}{A(P)}G_i.
\label{CM_2}
\end{equation}

If $P$ has vertices $V_1, \ldots, V_n$, with $V_i=(x_i, y_i)$, the center of mass $CM_2(P)$ can be expressed in terms of the coordinates of the vertices:

\begin{equation}
\begin{split}
CM_2(P)= \frac{1}{6A(P)} & \sum_{i=1}^{n} (x_iy_{i+1}-x_{i+1}y_i) (x_i+x_{i+1}, y_i+y_{i+1}).
\end{split}
\end{equation}

showing that $CM_2$ also does not depend on the triangulation.\\

Paper \cite{centers-poncelet} studies the locus of $CM_2$, however, there is a gap within the proof of the result. Proposition \ref{limit} of this paper recognizes this.\\

Before continuing, let us recall Poncelet closure theorem:

\begin{thm}[\textbf{Poncelet Porism}]
Let $\gamma, \Gamma$ be two nested conics in $R^2$, with $\gamma$ in the interior of $\Gamma$. Let $W_1 \in \Gamma$, and $\ell_1$ be one of the tangents to $\gamma$ passing through $W_1$: this line intersects $\Gamma$ in a new point that we call $W_2$. Repeat this construction, now starting from $W_2$, and considering the other tangent $\ell_2$ to $\gamma$ passing through $W_2$ to find $W_3$; and so on.\\
Suppose that for this particular choice of $W_1$, there is an $n\in \mathbb{N}$ such that $W_{n+1}=W_1$. \\
Then, for any other $V_1 \in \Gamma$ if $V_2, V_3, \ldots \in \Gamma$ are obtained by the same construction described above, it happens that $V_{n+1}=V_1$. See Figure \ref{Poncelet}.\\ 

The polygons with vertices $(V_1, \ldots, V_n)$ inscribed in $\Gamma$ and circumscribed about $\gamma$ are called \emph{Poncelet polygons}.
\end{thm}

Significant references for this classic result are \cite{Berger} and \cite{Flatto}.\\

\begin{figure}
\centering
\includegraphics[scale=0.27]{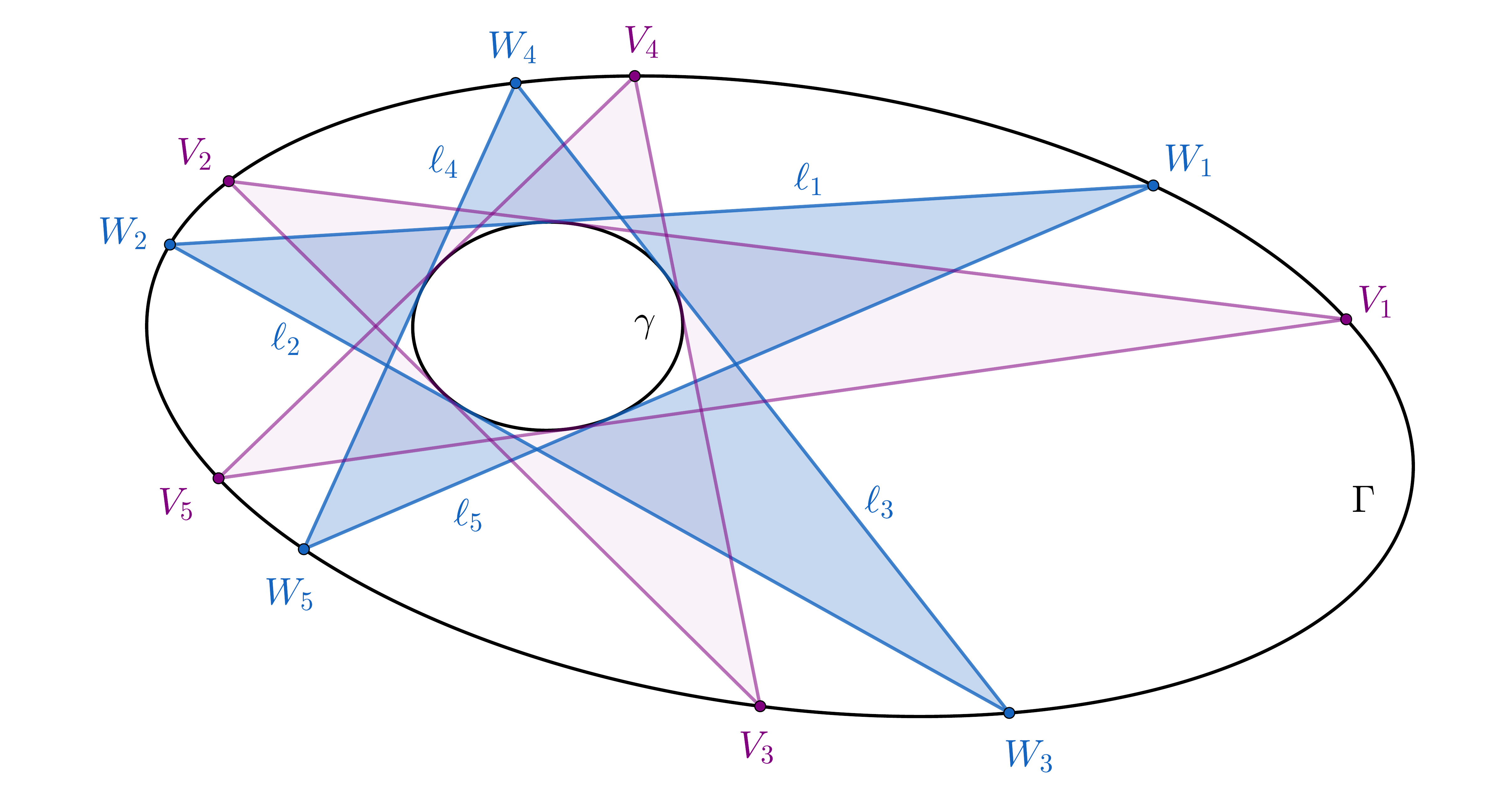}
\caption{Poncelet Theorem, with $n=5$.}
\label{Poncelet}
\end{figure}

Using the background set by Griffiths and Harris in \cite{G-H}, and Liouville's theorem, we show the following result:

\begin{thm}
Let $\gamma, \Gamma$ be a pair of conics that admit a $1$-parameter family of Poncelet $n$-gons $P_t$. The locus of the Circumcenter of Mass $CCM(P_t)$ is also a conic.
\end{thm}

The same techniques and tools provide a proof of a result found previously by Dan Reznik (shared via personal communication with S. Tabachnikov):

\begin{thm}
Let $\gamma$, $\Gamma$ be two concentric ellipses in general position admitting a $1$-parameter family of Poncelet $n$-gons. Given $P=(V_1, \ldots, V_n)$ one of these Poncelet $n$-gons, let $Q$ be a new polygon formed by the tangent lines to $\Gamma$ at $V_i$.\\
If $n$ is even, then $A(P)\cdot A(Q)$ stays constant within the Poncelet family.
\end{thm}

One can find several papers with results that share this flavor, like the ones exposed in \cite{Akopyan-billiards}, \cite{Bialy-dan}, \cite{Fierobe-circumcenter}, \cite{Reznik-ellipses}, \cite{Reznik-new-properties}, \cite{Reznik-ballet}, \cite{Reznik-intelligencer} and \cite{Romaskevich-incenter} to mention some.


\section{Poncelet polygons with zero area}

Following the spirit of \cite{G-H}, \cite{Romaskevich-incenter}, and \cite{centers-poncelet}, we complexify and projectivize the picture. That is, even though the objects we work with are initially $2$-real dimension objects, we allow the coordinates to take complex values and add the points at infinity; so from now on we will be working in $\mathbb{CP}^2:= \mathbb{C}^2 \cup {\cal L}_\infty $.

\begin{defi} 
Let $P$ be an oriented $n-$gon with vertices $V_1, \ldots , V_n \in \mathbb{C}^2$, where $V_i=(x_i, y_i)$. The \emph{area} of $P$ is

\begin{equation}\label{area}
A(P)=\frac{1}{2}\sum_{i=1}^{n}(x_iy_{i+1}-x_{i+1}y_i).
\end{equation}

By convention, we set $x_{n+1}:= x_1$, $y_{n+1}:=y_1$, and in general, subindices are taken modulo $n$. 
\end{defi}

Notice the area of a polygon is invariant under equiaffine transformations but not under projective transformations.\\

Recall that, if $\Gamma \subset \mathbb{CP}^2$ is a $C^1$ curve and $p\in \Gamma$ then $T_p\Gamma$ denotes the line tangent to $\Gamma$ that passes through $p$. Similarly $T\Gamma= \{T_p\Gamma : p\in \Gamma\}$.\\

Given two conics $\gamma$, $\Gamma \subset \mathbb{CP}^2$ we say that $\gamma, \Gamma$ are in \emph{general position} if they intersect transversally, i.e. at every point in the intersection $p\in \gamma \cap \Gamma$, we have $T_p\gamma \cap T_p\Gamma=\{p\}$. See Figure \ref{not-transverse} for a counterexample.\\

\begin{figure}
\centering
\includegraphics[scale=0.3]{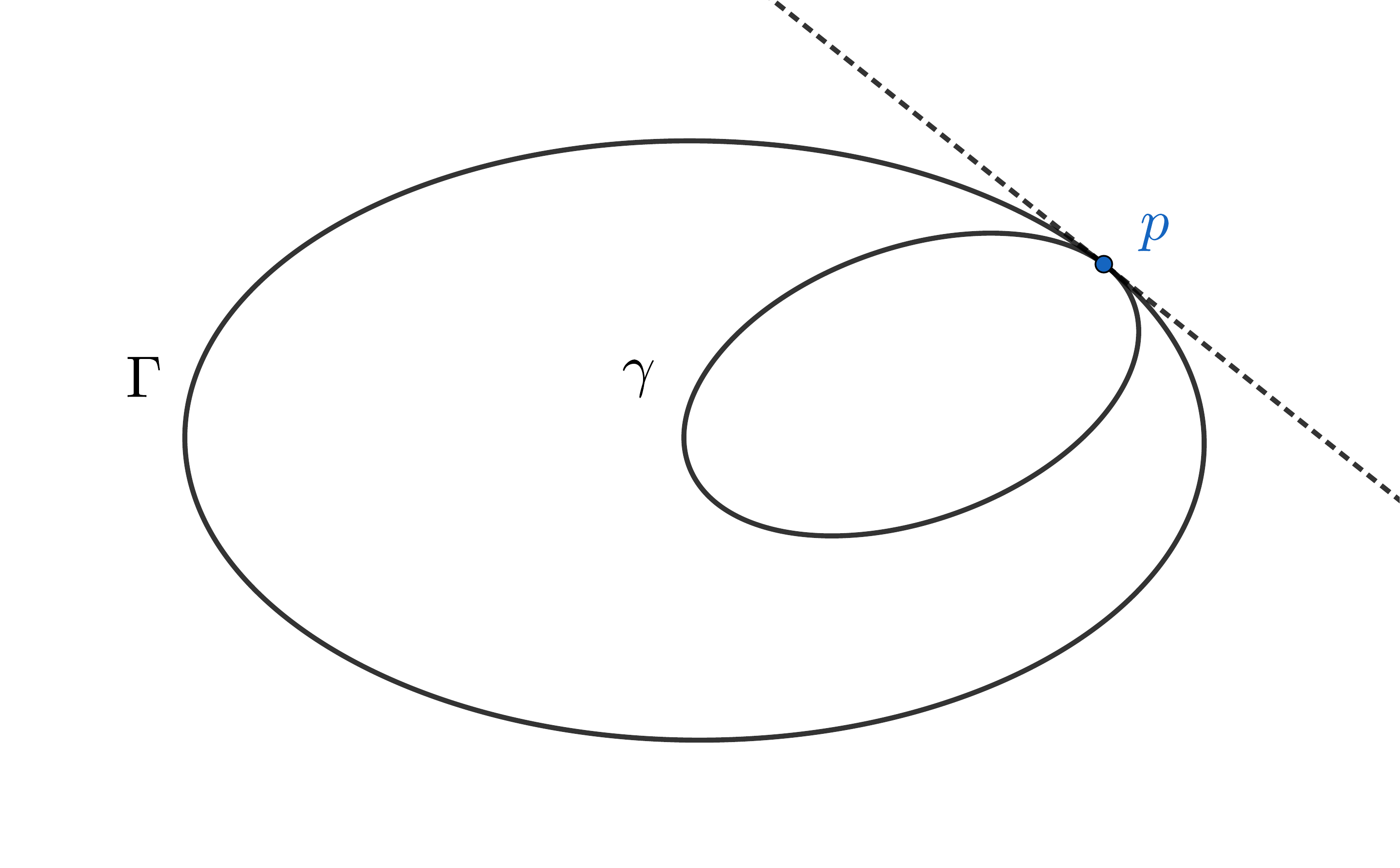}
\caption{The conics $\gamma$ and $\Gamma$ are not in general position, since both tangents $T_p\gamma$ and $T_p \Gamma$ coincide.}
\label{not-transverse}
\end{figure}

\begin{defi} 
Let $\gamma, \Gamma \subset \mathbb{CP}^2$ be a pair of conics that admit a $1$-parameter family of Poncelet $n$-gons. A Poncelet polygon of this family is called \emph{degenerate} if at least one of the vertices $V_i$ satisfy either $V_i\in \gamma \cap \Gamma$ or $T_{V_i}\Gamma \in T\gamma$. This is illustrated in Figure \ref{degenerate}.\\ 
Otherwise, a Poncelet polygon is called \emph{non-degenerate}.
\end{defi}

\begin{figure}
\centering
\includegraphics[scale=0.25]{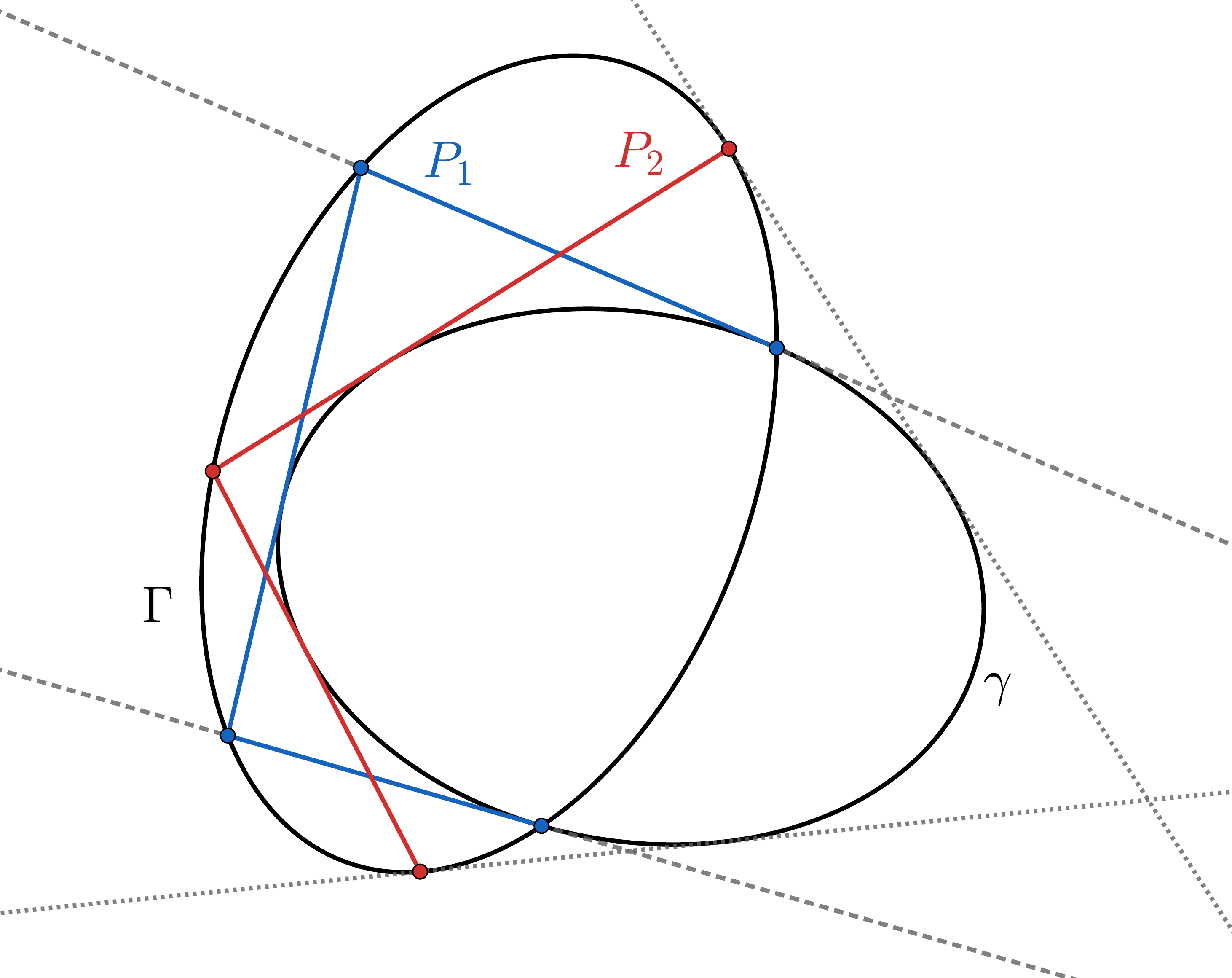}
\caption{The conics $\gamma, \Gamma$ admit a $1$-parameter family of Poncelet hexagons.
The degenerate polygon $P_1$ illustrates the case when one vertex belongs to the intersection $\gamma \cap \Gamma$, while the degenerate polygon $P_2$ corresponds to the case $T_{V_i}\Gamma \in T\gamma$.}
\label{degenerate}
\end{figure}

Notice that if $P$ is a degenerate polygon then:
\begin{itemize}
\item If $V_i \in \gamma \cap \Gamma$, then $V_{i-k}=V_{i+k}$ for all $k$. Intuitively, the polygon ``bends'' at $V_i$.
\item If $T_{V_i}\Gamma \in T\gamma$, then $V_i=V_{i+1}$, and $V_{i-k}=V_{i+1+k}$ for all $k$. The vertices $V_i$ and $V_{i+1}$ ``glue together''.
\end{itemize}
Figure \ref{even} shows a picture of Poncelet polygons that ``are about to collapse''.\\

\begin{figure}
\centering
\includegraphics[scale=0.25]{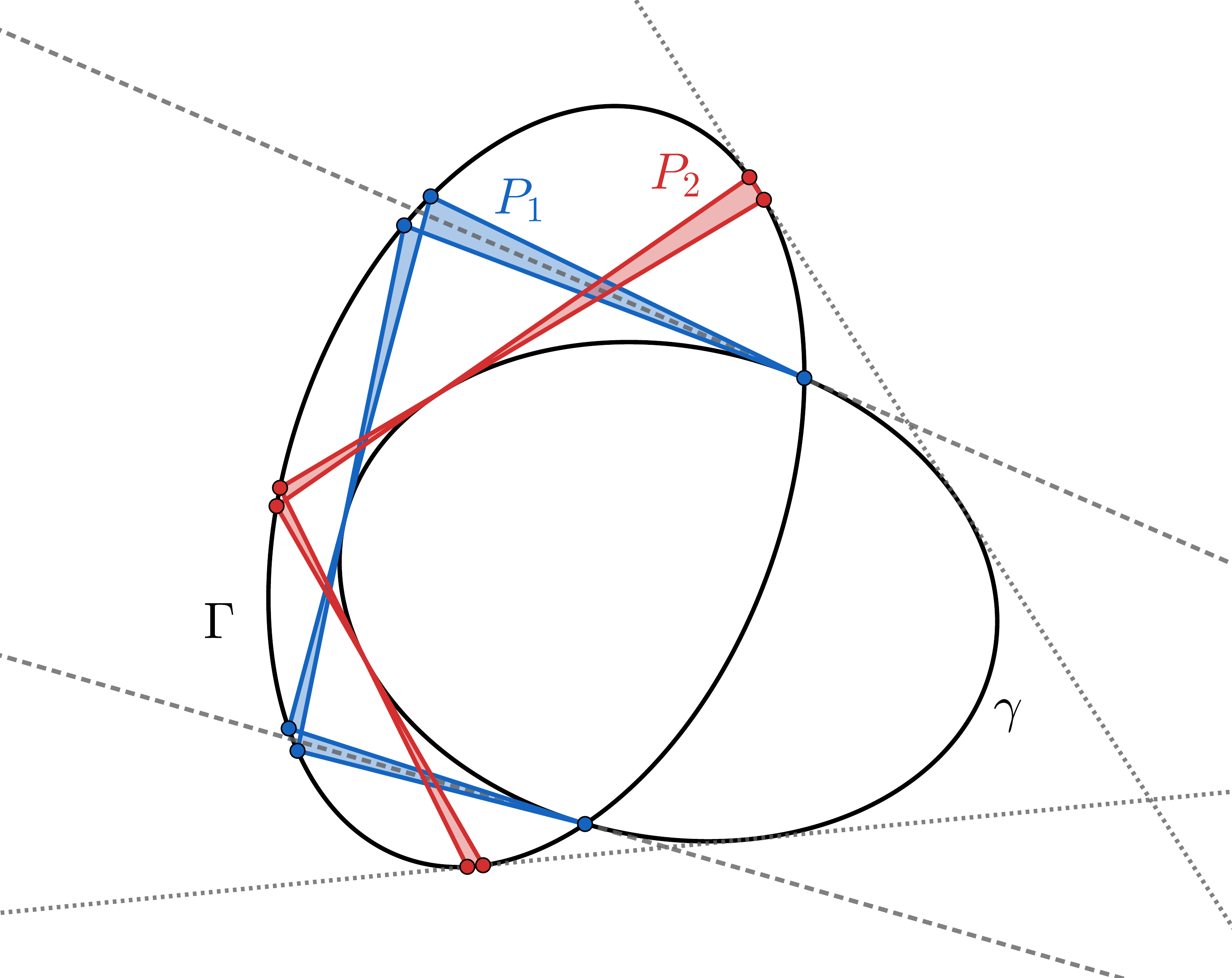}
\figcaption{A Poncelet polygon can degenerate if either two sides coincide or if two vertices coincide. $P_1$ is ``bending'' and $P_2$ has a pair of vertices that are ``gluing together''.}\label{even}
\end{figure}

From this remark, we get the following observations regarding degenerate polygons.

\begin{lem}\label{zero-degenerate}
 If $\gamma$, $\Gamma \subset \mathbb{CP}^2$ admit a $1$-parameter family of Poncelet $n$-gons $P_t$ inscribed in $\Gamma$ and circumscribed about $\gamma$, then  
\begin{enumerate}
\item If $P$ is a degenerate Poncelet polygon, with finite vertices, then $A(P)=0$.
\item If $\gamma, \Gamma$ are conics in general position, then there are $4n$ different degenerate polygons.
\end{enumerate}
\end{lem}

\textbf{Note:} The order of the label in the vertices matter: that is, if $P=(V_1, \ldots, V_n)$ and $P'=(V'_1,\ldots, V'_n)$ satisfy $V_i=V'_{i+c}$ for all $i$, with $c$ constant, then we consider $P$ and $P'$ to be \textbf{equal} if and only if $\boldsymbol{c \equiv 0 (\textbf{mod } n)}$.\\
 
\begin{proof}
The first observation is straightforward.\\
For the second part, we divide in two cases:

\begin{itemize}
\item \textbf{n even.} 
Recall $P=(V_1, \ldots, V_n)$ is degenerate if one of the following happens: 

	\begin{itemize}
	\item \textit{$V_i \in \gamma \cap \Gamma$ for some index $i$.} \\
	Since $V_{i-k}=V_{i+k}$, in particular one has $$V_{i+\frac{n}{2}-1}=V_{i-\frac{n}{2}+1}=V_{i+\frac{n}{2}+1},$$ (the last equality comes from the congruence mod $n$ in the indices),  which means $V_{i+\frac{n}{2}}$ has a unique tangent to $\gamma$, thus $V_{i+\frac{n}{2}} \in \gamma \cap \Gamma$. \\
	If $P$ is one of these degenerate polygons, then it has $n/2+1$ different vertices: two in the intersection $\gamma\cap\Gamma$ and the rest repeating twice. In total, there are $n$ different manners one can list the vertices of this polygon in the correct order: if one starts from a vertex in the intersection, there is just one way to list the rest; but if one starts with any double vertex $V_i$, then the next vertex can be either $V_{i+1}$ or $V_{i-1}$.\\
	Since $\gamma, \Gamma$ are conics in general position, by B\'ezout's theorem, $|\gamma \cap \Gamma|=4$. Therefore we have $2n$ of these degenerate polygons.\\
	
	\item \textit{$T_{V_i}\Gamma \in T\gamma$ for some $i$.} \\
	Then $V_i=V_{i+1}$, and $V_{i-k}=V_{i+1+k}$. In particular $$V_{i-\frac{n}{2}}=V_{i+\frac{n}{2}+1}=V_{i-\frac{n}{2}+1},$$ meaning $T_{V_{i-\frac{n}{2}}}\Gamma \in T\gamma$. \\
	As before, if $P$ is one of these degenerate polygons, it has $n/2$ different vertices, all of them double vertices, with exactly two of these vertices satisfying $T_{V_i}\Gamma \in T\gamma$. Similarly, there are $n$ different ways one can list the vertices of $P$ in the correct order.\\
	In the dual space, (the space of lines) $T\gamma$ and $T\Gamma$ are conics as well, so $|T\gamma \cap T\Gamma|=4$. There are then $2n$ degenerate polygons of this type.\\
	
	In the case where $\gamma$ and $\Gamma$ are in general position, these two cases are disjoint, and so we have $4n$ degenerate Poncelet polygons in total.
	\end{itemize}
 
\item \textbf{n odd.} Applying a similar analysis to the previous case, a polygon $P=(V_1,\ldots V_n)$ is degenerate if and only if exactly one vertex $V_i$ belongs to the intersection $\gamma \cap \Gamma$ and $T_{V_{i+\frac{n-1}{2}}}\Gamma \in T\gamma$ (see Figure \ref{odd}). Then there are $4n$ degenerate polygons.

\begin{figure}
\centering
\includegraphics[scale=0.27]{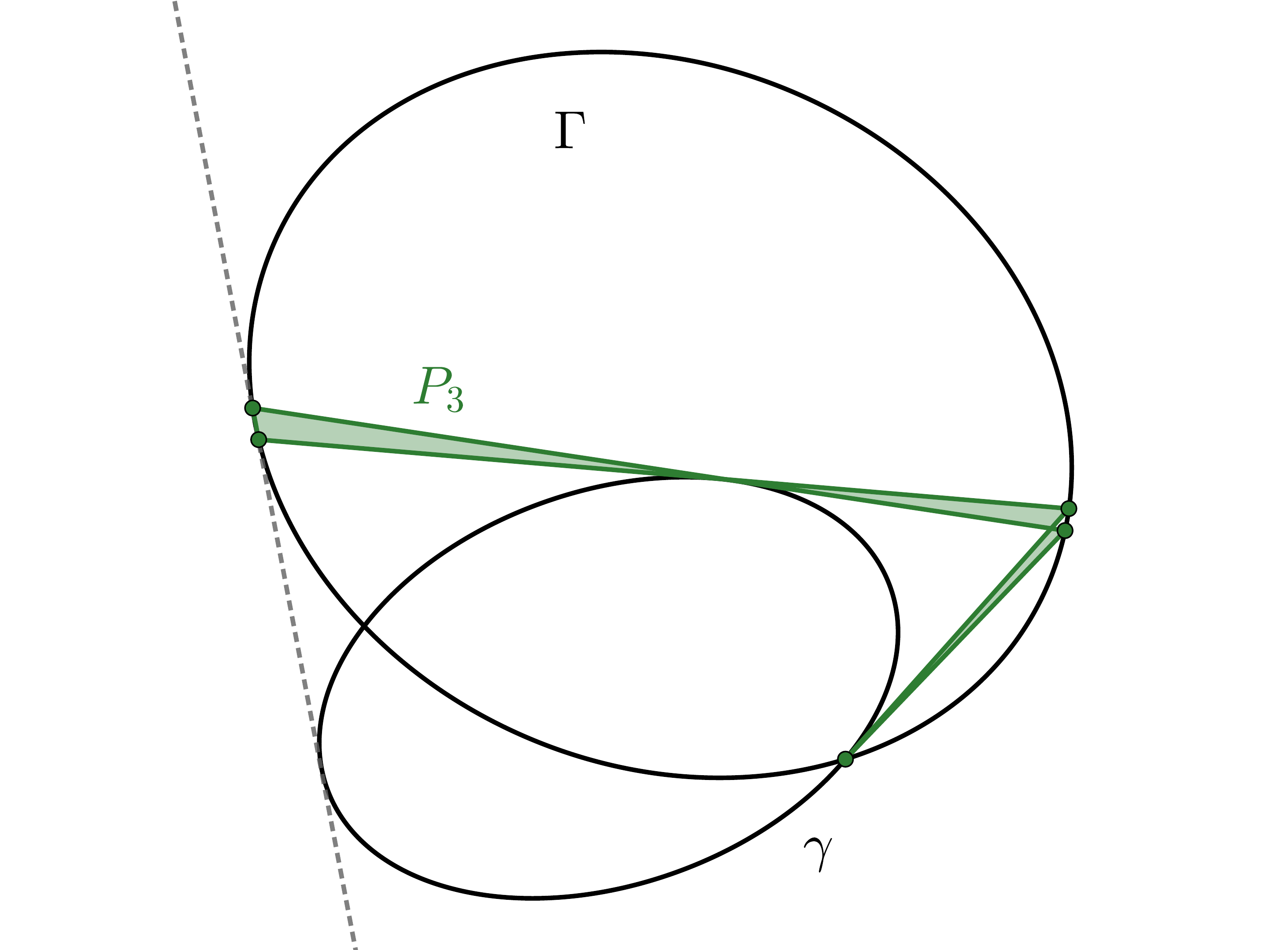}
\figcaption{If $n$ is odd (in this case $n=5$), then each degenerate Poncelet polygon has one vertex where the polygon ``bends'' (when $V_i\in \gamma \cap \Gamma$) and two vertices that ``glue together'' (if $T_{V_j}\Gamma \in T\gamma$, then $V_j=V_{j+1}$). }\label{odd}
\end{figure}

\end{itemize}
\end{proof}

\begin{thm}\label{jefe} 
Let $\gamma, \Gamma \subset \mathbb{CP}^2$ be a pair of conics in general position that admit a $1$-parameter family of Poncelet $n-$gons, with vertices in $\Gamma$ and sides tangent to $\gamma$. Suppose there is a non-degenerate Poncelet polygon P such that $A(P)=0$. Then every Poncelet polygon of the family has area zero.
\end{thm}

Before going over the proof, let us examine the particular case when $n=4$.

\begin{lem}\label{bowtie} A quadrilateral $Q =(V_1, V_2, V_3, V_4)$ has area zero if and only if the diagonals $V_1V_3$ and $V_2V_4$ are parallel.
\end{lem}

\begin{proof}
Notice $A(V_1,V_2,V_3,V_4)=A(V_1,V_2,V_3)+A(V_1,V_3,V_4) = A(V_1,V_2,V_3)- A(V_1,V_4, V_3)$.\\

The areas $A(V_1, V_2, V_3)$ and $A(V_1, V_4, V_3)$ are equal if and only if $V_1V_3$ is parallel to $V_2V_4$. This proves the result.

\end{proof}

\begin{prop}
Let $\gamma$ and $\Gamma$ be conics admitting a $1$-parameter family of Poncelet quadrilaterals $P_t$. If there is one non-degenerate Poncelet quadrilateral with area zero, then all the Poncelet quadrilaterals have area zero.\\
Moreover, the locus of each $CCM(P_t)$ and $CM_2(P_t)$ is a point at infinity, with $$CM_2(P_t)=\{[a:b:0]\} \text{ and } CCM(P_t)=\{[-b:a:0]\},$$ for some $a,b \in \mathbb{C}$ not simultaneously $0$.
\end{prop}

\begin{proof}
Recall that, if $\gamma, \Gamma$ admit a $1$-parameter family of Poncelet quadrilaterals $P_t$, then the diagonals of $P_t$ for any $t$, intersect at a fixed point. This can be proved by taking a projective transformation that makes $\Gamma$ and $\gamma$ concentric and using the fact that the Poncelet map commutes with the reflection with respect to the center of $\gamma$ and $\Gamma$. \\

If one non-degenerate Poncelet quadrilateral has area zero, by Lemma \ref{bowtie} the diagonals must be parallel and so they intersect at some point at infinity, say $[a:b:0]$. The diagonals of any other Poncelet quadrilateral then also intersect at $[a:b:0]$ and hence are parallel (see Figure \ref{PonceletBowtie}).\\

Recall that if $P$ is a quadrilateral then $CCM(P)$ lies in the intersection of the perpendicular bisectors of the diagonals. Therefore $CCM(P_t)=\{[-b:a:0]\}$. Similarly $CM_2(P)$ lies in the intersection of the line connecting the centroids of $V_1 V_2 V_4$ and $V_2 V_3 V_4$, and the line connecting the centroids of $V_1 V_2 V_3$ and $V_1 V_3 V_4$. In this case, these lines are both parallel to the diagonals, and so $CM_2(P_t)=\{[a:b:0]\}$. 
\end{proof}

\begin{figure}
\centering
\includegraphics[scale=0.3]{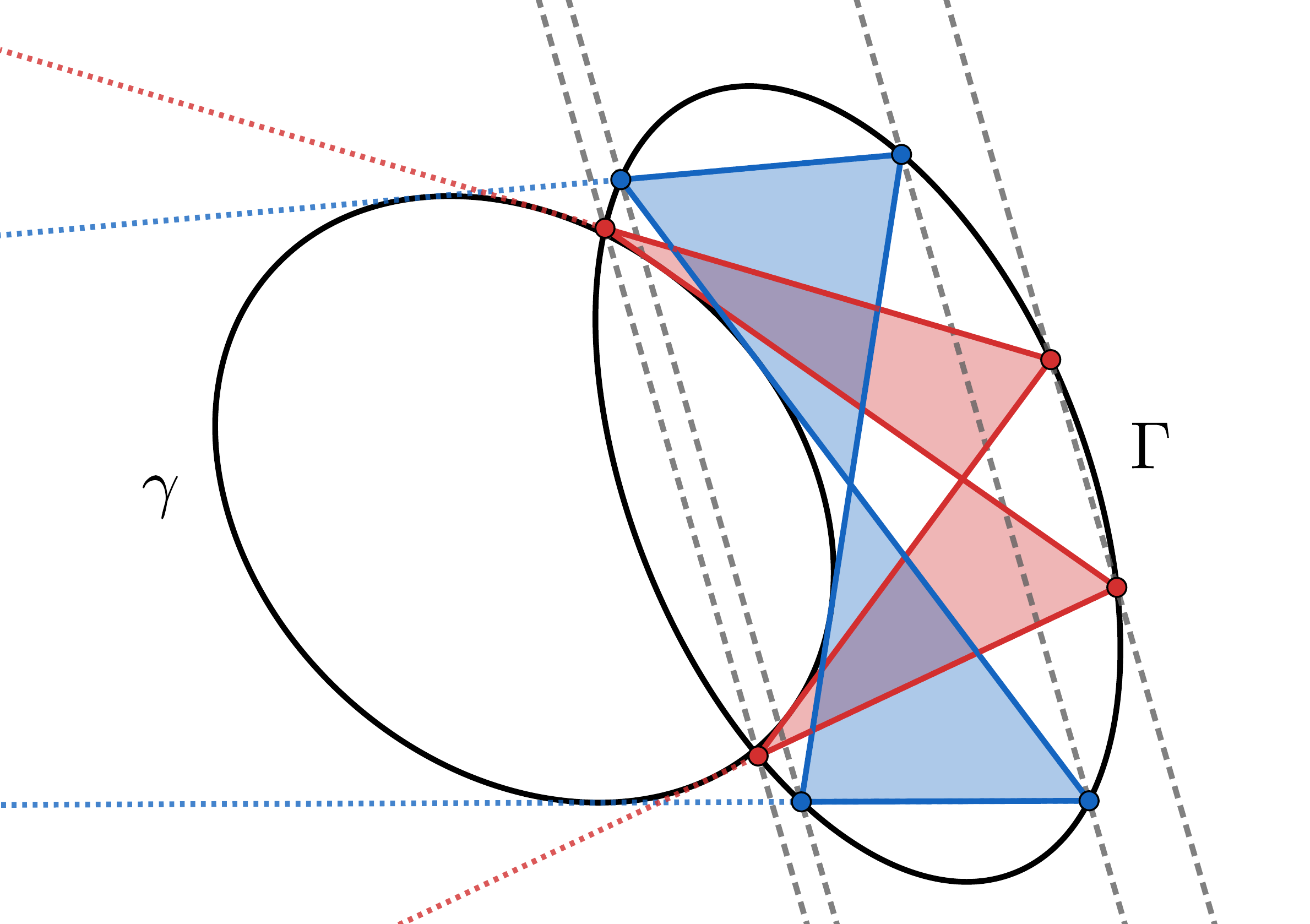}
\caption{If the diagonals of one non-degenerate Poncelet quadrilateral are parallel, then the diagonals of any other Poncelet quadrilateral in the family are also parallel.}
\label{PonceletBowtie}
\end{figure}

\begin{proof}
\textit{(of Theorem \ref{jefe})} The framework of Griffiths and Harris exposed in \cite{G-H} constitute the foundations of this proof. Papers like  \cite{centers-poncelet} make use of their approach to prove similar results.\\
Recall that any conic in $\mathbb{CP}^2$ is isomorphic to $\hat{\mathbb{C}}$ via stereographic projection from any point. \\
Consider the set of flags $${\cal E}=\{(p, \ell) : p\in \Gamma, p \in \ell, \ell \in T\gamma \}.$$
The projection $\pi: {\cal E} \rightarrow \Gamma$ given by $\pi(p, \ell)=p$ is a two-to-one map, with four branch points: $(q, \ell)$ where $q\in \gamma \cap \Gamma$. Computing the Euler characteristic, one obtains that ${\cal E}$ is topologically a torus. Moreover ${\cal E}\subset \Gamma\times T\gamma$ is an elliptic curve.\\
For the Poncelet map, consider the involutions $$ \sigma(p, \ell)=(p',\ell),\hspace{1cm} \tau(p', \ell)=(p', \ell').$$
See Figure \ref{involution}.\\

\begin{figure}
\centering
\includegraphics[scale=0.3]{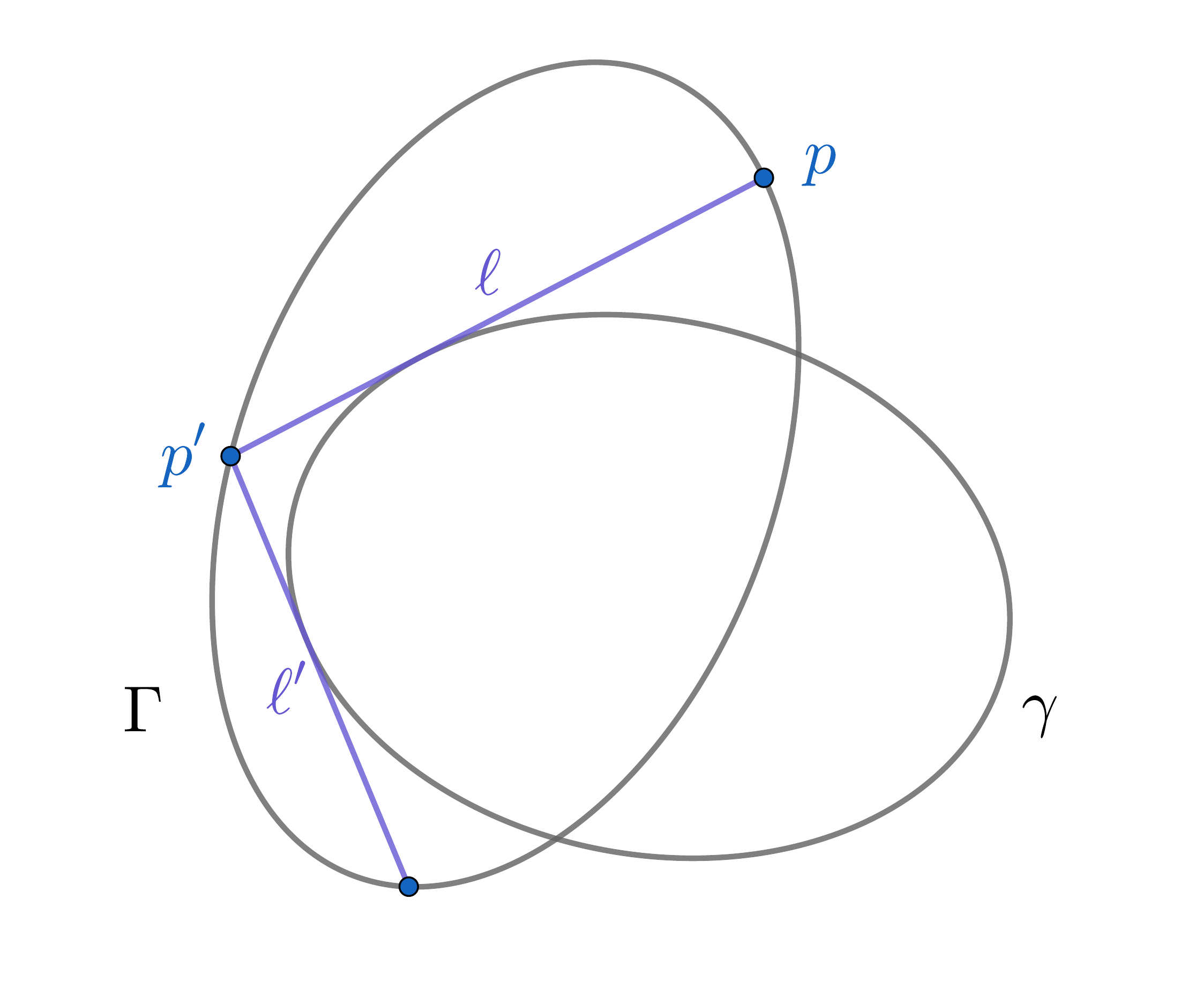}
\caption{ $\sigma(p, \ell)=(p',\ell)$, and $\tau(p', \ell)=(p', \ell')$.}
\label{involution}
\end{figure}

The Poncelet map is defined as $T=\tau\circ\sigma$. Since $\tau$ and $\sigma$ are involutions, ${\cal E}$ admits a parameter $t$ in which the Poncelet map is a translation of the torus, $T(t)=t+c$, with $c$ some constant.\\
 
Now, let's examine the area function for Poncelet polygons. Notice that each flag $(p, \ell)$ gives rise to a unique Poncelet polygon with orientation.
Also, the group generated by $\sigma$ and $\tau$ is precisely the dihedral group $D_n$. If $Z_n\subset D_n$ is the cyclic group of order $n$, observe that $A: {\cal E} \rightarrow \hat{\mathbb{C}}$ is a $Z_n$-invariant meromorphic function, and since it is defined over a torus, $A$ is an elliptic function.\\

For conics $\gamma, \Gamma$ in general position, the area has a simple pole exactly when one of the vertices goes to infinity. Then, $A$ is an elliptic function of order $4n$, with $2n=|D_n|$ simple poles coming from each one of the points of $\Gamma$ at ${\cal L}_\infty$. The poles are simple because $\gamma$ and $\Gamma$ are in general position.\\

By Lemma \ref{zero-degenerate}, $A$ has at least $4n$ zeroes. If there was at least one non-degenerate Poncelet polygon with area zero, then this would contradict the fact that every non-constant elliptic function of order $m$ has exactly $m$ zeros. This implies that $A$ must be constant.
\end{proof}


\section{Limit of $CCM$ and $CM_2$ for Poncelet polygons}

From now on, we focus only on the case when not all Poncelet polygons have area zero. \\

Consider $P=(V_1, \ldots , V_n)$ a degenerate polygon with finite vertices. Thanks to Theorem \ref{jefe} we know we can find a $1$-parameter family of non-zero-area Poncelet polygons $P_t=(W_0(t), \ldots , W_n(t))$ approaching to $P$ as $t$ goes to zero. Since the area of each $P_t$ is not zero, the Center of Mass $CM_2(P_t)$ and the Circumcenter of Mass $CCM(P_t)$ are well defined. \\

Our goal for this section is to show that the limits
$$\lim_{t\rightarrow 0} CM_2(P_t) \hspace{0.3cm}\text{ and } \hspace{0.3cm} \lim_{t\rightarrow 0}CCM(P_t),$$
exist.\\
 
The way the vertices $W_i(t)$ approach to $V_i$ depends on the parity of $n$ and $V_i$ itself:

\begin{itemize}
\item[(a)] \textit{Around ``bending'' vertices:} If $V_i \in \gamma \cap \Gamma$, then $W_{i+1}(t)$ and $W_{i-1}(t)$ approach to $V_{i+1}$, as shown in Figure \ref{deg(a)}: \\
	\begin{itemize}
	\item[$\bullet$] $W_{i}(t)=V_i+\overrightarrow{k_i}t+O(t^2)$,
	\item[$\bullet$] $W_{i+1}(t)= V_{i+1} + \overrightarrow{k}_{i+1}t+O(t^2)$ and
	\item[$\bullet$] $W_{i-1}(t)= V_{i+1} + \lambda_{i-1}\overrightarrow{k}_{i+1}t+O(t^2)$ with $\lambda_{i-1}\neq 1$ (since $W_{i+1}(t)\neq W_{i-1}(t)$ for all $t\neq 0$).
	\end{itemize}
	
	\begin{center}
	\includegraphics[scale=0.42]{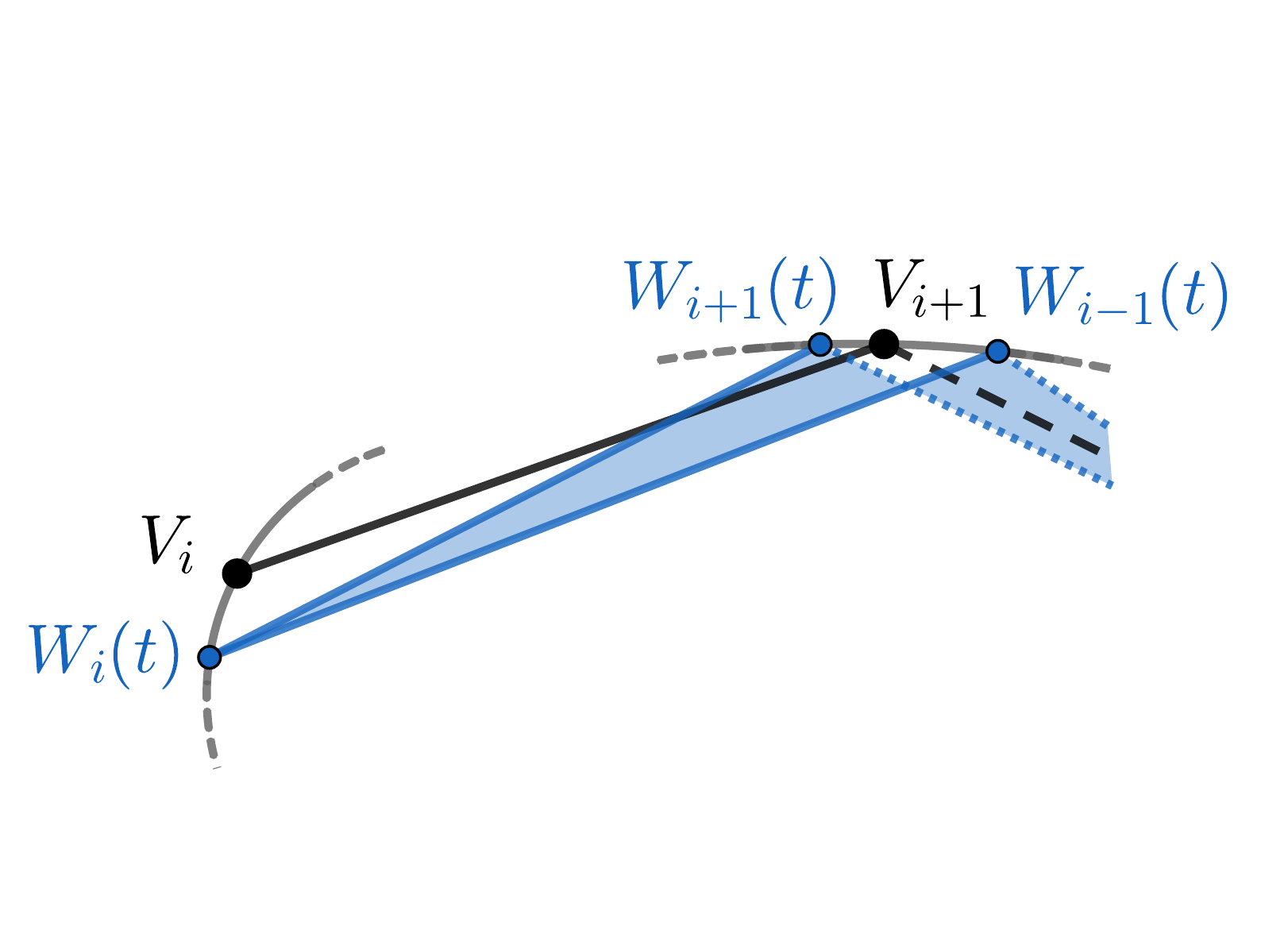}
	\figcaption{Behavior around $V_i\in \gamma \cap \Gamma$.}
	\label{deg(a)}
	\end{center}	 
	 
\item[(b)] \textit{Around ``gluing'' vertices:} If $T_{V_i}\Gamma \in T\gamma$, and $V_{i}=V_{i+1}$ then the vertices $W_{i}(t), W_{i+1}(t)$ are of the form:
	\begin{itemize}
	\item[$\bullet$] $W_i(t)= V_i + \overrightarrow{k_i}t +O(t^2)$ and
	\item[$\bullet$] $W_{i+1}(t)= V_i + \lambda_i\overrightarrow{k_i}t+O(t^2)$ with $\lambda_{i}\neq 1$ (since $W_i(t)\neq W_{i+1}(t)$ for all $t\neq 0$).
	\end{itemize}
	Figure \ref{deg(b)} illustrates the situation.
	
	\begin{center}
	\includegraphics[scale=0.42]{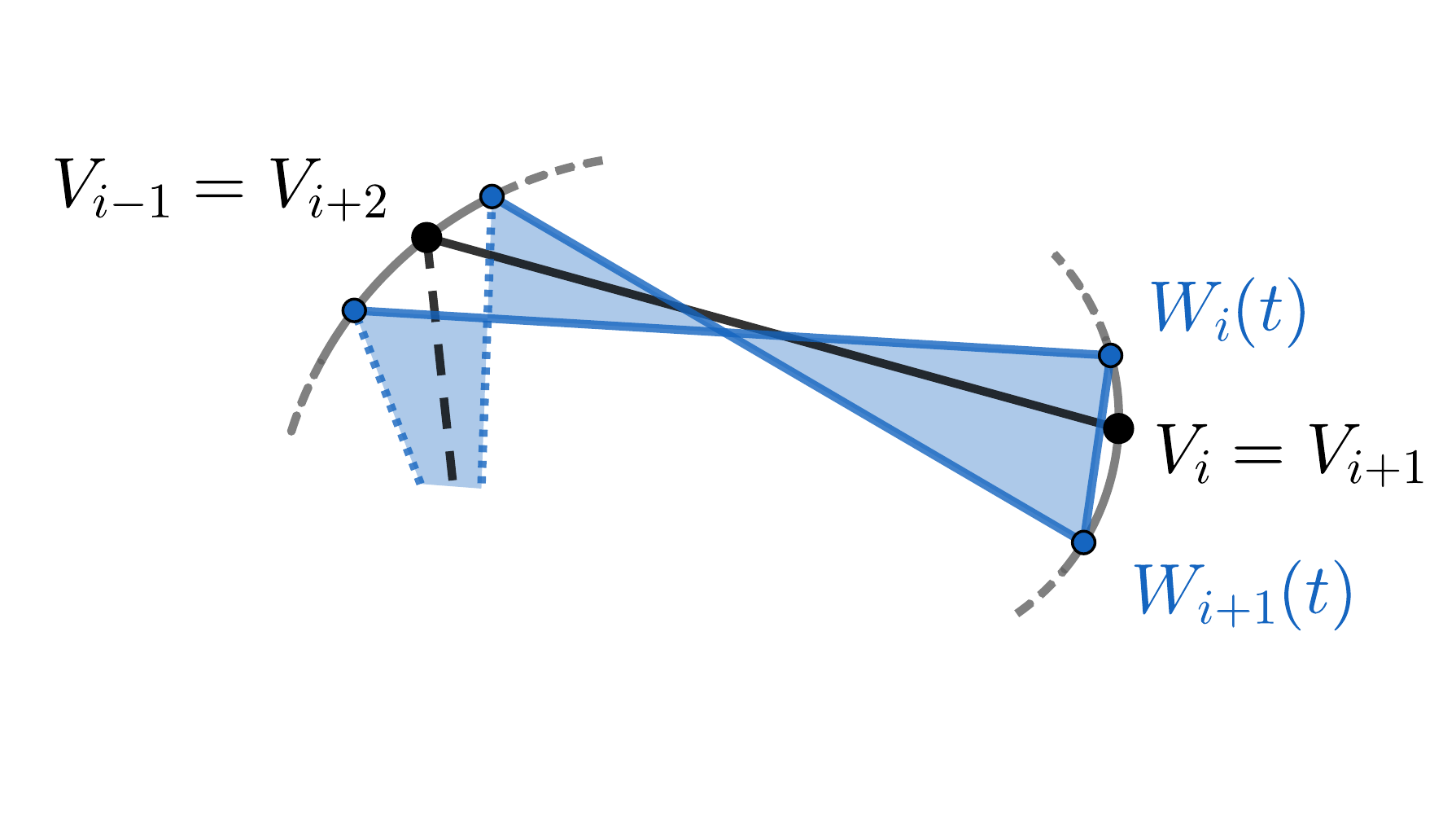}
	\figcaption{Behavior around $V_i$, when $T_{V_i}\Gamma \in T\gamma$.}
	\label{deg(b)}
	\end{center}
	 
\item[(c)] \textit{For the rest of the vertices:} If $V_i=V_j$ for some $i\neq j$ non-consecutive, then $V_{i+1}=V_{j-1}$.\\
Hence:
	\begin{itemize}
	\item[$\bullet$] $W_i(t)= V_i + \overrightarrow{k_i}t +O(t^2)$,
	\item[$\bullet$] $W_j(t)= V_i + \lambda_i\overrightarrow{k_i}t +O(t^2)$ with $\lambda_i\neq 1$, 
	\item[$\bullet$] $W_{i+1}(t)= V_{i+1} + \overrightarrow{k}_{i+1}t +O(t^2)$ and
	\item[$\bullet$] $W_{j-1}(t)= V_{i+1} + \lambda_{i+1}\overrightarrow{k}_{i+1}t +O(t^2)$ with $\lambda_{i+1}\neq 1$.
	\end{itemize}

	\begin{center}
	\includegraphics[scale=0.42]{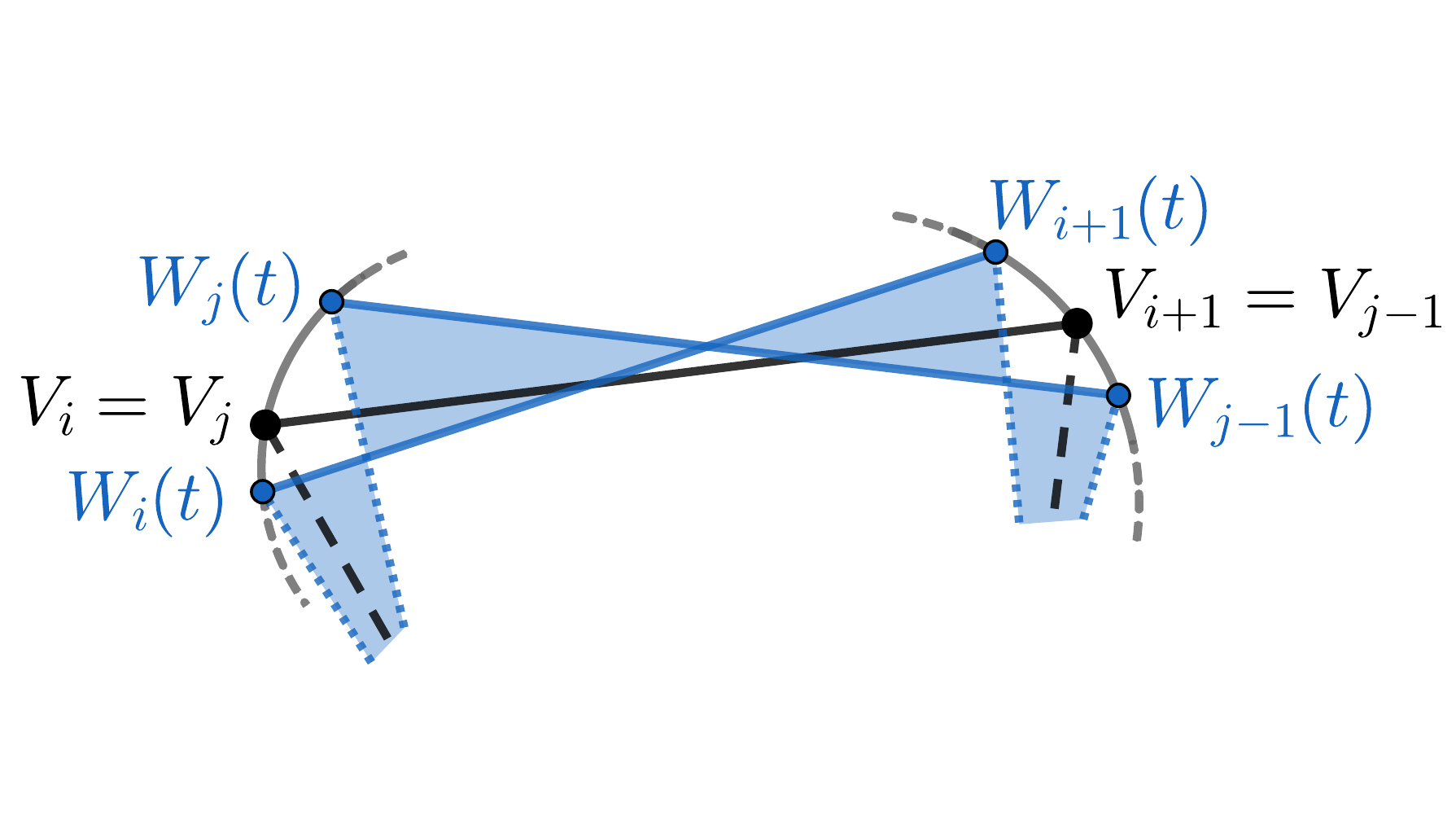}
	\figcaption{Behavior for the rest of the vertices of $P_t$.}
	\label{deg(c)}
	\end{center}

\end{itemize}

\begin{prop} Let $\gamma, \Gamma$ be a pair of conics that admit a $1$-parameter family of Poncelet polygons with area not constant zero. If $P$ is a degenerate polygon with finite vertices, and $P_t$ a $1$-parameter family of non-degenerate Poncelet polygons such that $\lim_{t\rightarrow 0}P_t = P$  then $$\lim_{t\rightarrow 0} CCM(P_t) \hspace{0.3cm}\text{ and } \hspace{0.3cm} \lim_{t\rightarrow 0}CM_2(P_t)$$ exist.
\label{limit}
\end{prop}

\begin{proof}
Regardless of the parity of $n$, we can triangulate the polygons $P_t$ into triangles $(W_{j_1}(t), W_{j_2}(t), W_{j_3}(t))$ such that 
$$\lim_{t\rightarrow 0} W_{j}(t)= \lim_{t\rightarrow 0} W_{j'}(t)=V_i$$ 
for exactly two indices $j, j' \in \{j_1, j_2, j_3\}$ and for some vertex $V_i$ of $P$.\\

\begin{itemize}
\item[(a)] If $V_i\in \gamma \cap \Gamma$, take $(W_{i-1}(t), W_i(t), W_{i+1}(t))$ as part of the triangulation. This is shown in Figure \ref{triangulation} (left).\\
\item[(b)] If $V_i=V_j$ for some $i\leq j$, add the triangles $(W_i(t), W_{i+1}(t), W_j(t))$ and $(W_{j-1}(t), W_j(t), W_{i+1}(t))$ to the triangulation. This is illustrated in Figure \ref{triangulation} (right).\\
\end{itemize}

\begin{figure}
\centering
\includegraphics[scale=0.35]{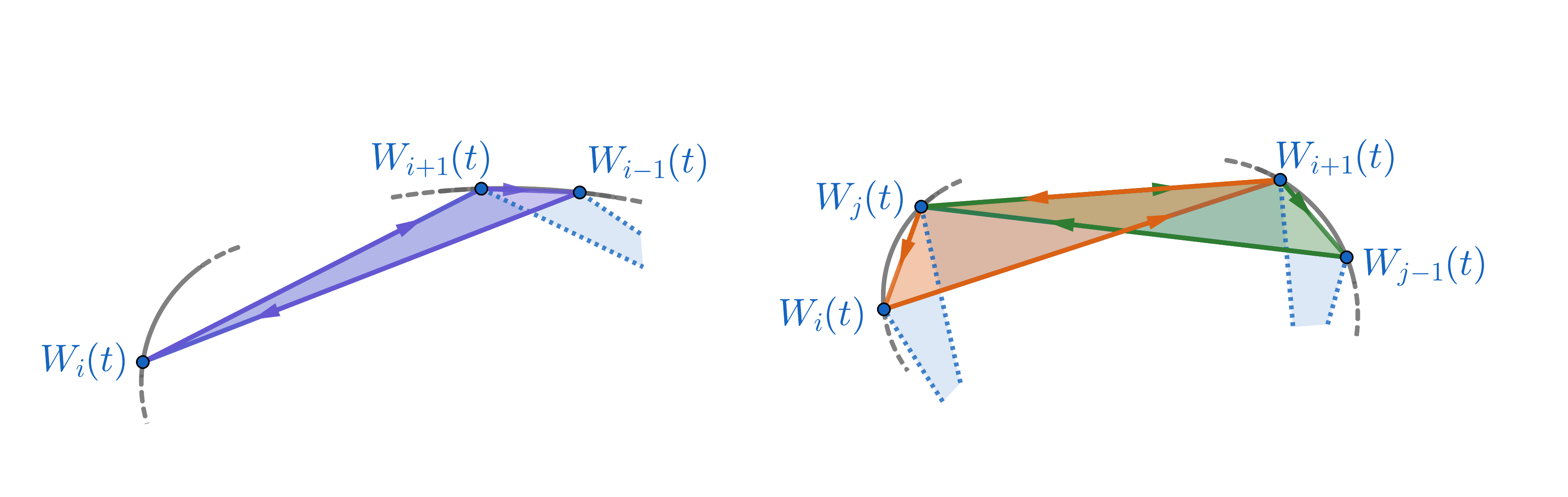}
\caption{Triangulation of $P_t$. The figure in the left shows the triangulation around vertices in the intersection $\gamma \cap \Gamma$. The figure in the right shows the triangulation for the rest of the vertices.}
\label{triangulation}
\end{figure}

\clearpage

\underline{\textit{Claim:}}
\begin{enumerate}
	\item The area of each triangle tends to zero linearly with respect to $t$. 
	\item The limit of the circumcenter of each triangle exists.
	\item The limit of the centroid of each triangle exists.
\end{enumerate}

Before proving the claim, notice that if it holds, then we are done. Consider the triangulation $\{T_1(t), \ldots T_{n-2}(t)\}$ of $P_t$ described above. If the area of each triangle $A_i(t)$ tends to zero linearly with respect to $t$ and the area of $P_t$ is not zero, then the area $A(P_t)$ also tends to zero linearly with respect to $t$, and so $\lim_{t\rightarrow 0} \frac{A_i(t)}{A(P_t)}$ exists.\\
Denote by $C_i(t)$ the circumcenter of $T_i(t)$. Then the second part of the claim secures that $\lim_{t\rightarrow 0} C_i(t)$ exists for each subindex $i$. Hence, 
$$\lim_{t\rightarrow 0} \sum_{i=1}^{n-2} \frac{A_i(t)}{A(P_t)} C_i(t)$$
exists.\\

Similarly, if $G_i(t)$ denotes the centroid of $T_i(t)$, then the third part of the claim ensures that $\lim_{t\rightarrow 0}G_i(t)$ exists for all $i$. Thus,
$$\lim_{t\rightarrow 0} \sum_{i=1}^{n-2} \frac{A_i(t)}{A(P_t)} G_i(t)$$
also exists.
\\

\textit{Proof of the claim:}\\
Consider one triangle $T(t)=(W_{j_1}(t), W_{j_2}(t), W_{j_3}(t))$ of this triangulation. Without lost of generality, suppose $$\lim_{t\rightarrow 0} W_{j_1}(t)=\lim_{t\rightarrow 0} W_{j_2}(t) = V_i.$$
That is:

\begin{itemize}
	\item $W_{j_1}(t)=V_i+ \overrightarrow{k_i}t+O(t^2),$
	\item $W_{j_2}(t)=V_i+ \lambda \overrightarrow{k_i}t+O(t^2),$
	\item $W_{j_3}(t)=V_{i\pm 1}+ \overrightarrow{k}_{i\pm 1}t+O(t^2).$
\end{itemize}

where $\lambda \neq 1$, and $\overrightarrow{k_i}$ is a vector tangent to $\Gamma$ at $V_i$, as shown in Figure \ref{local}: 

\begin{center}
\includegraphics[scale=0.35]{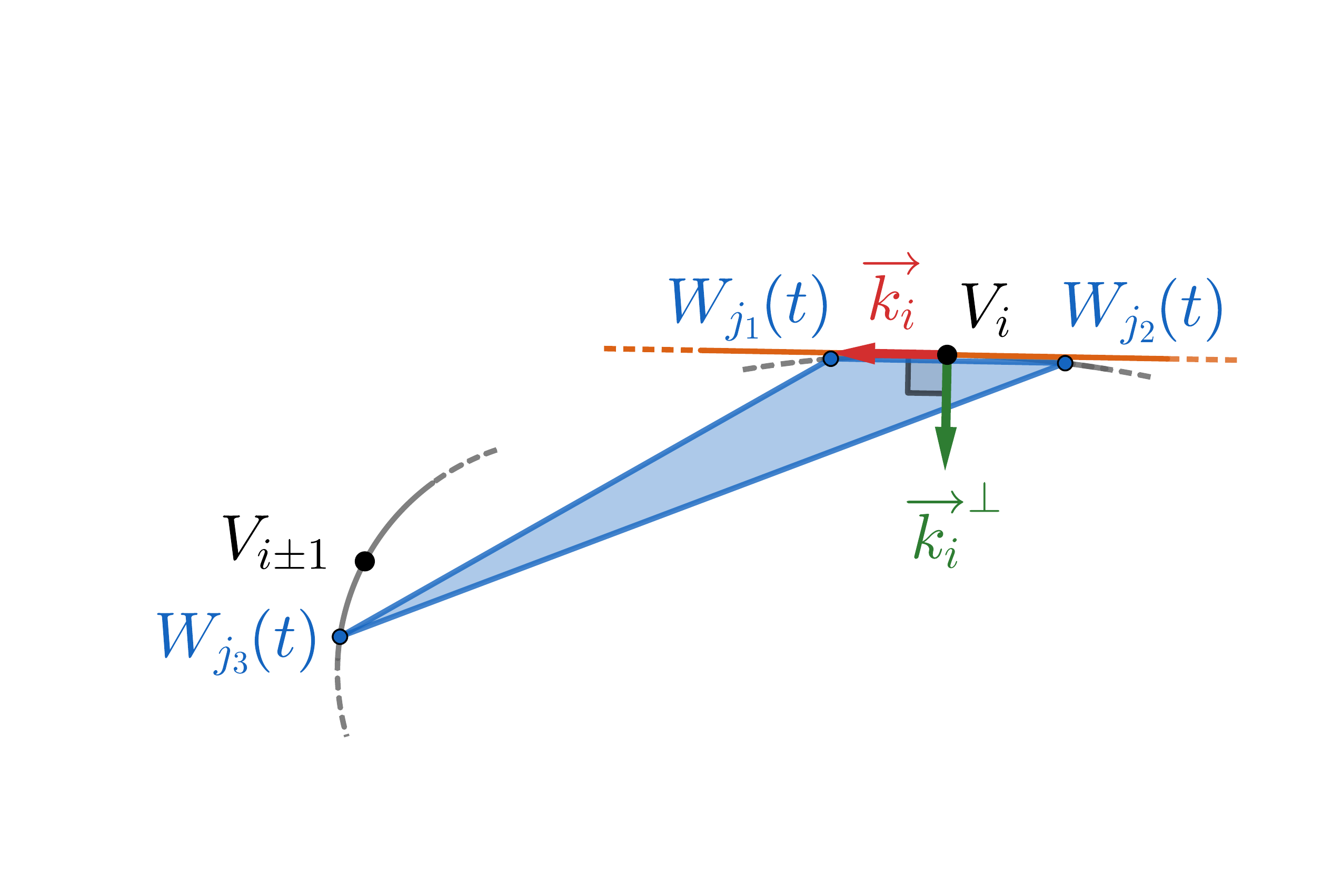}
\figcaption{The triangle $(W_{j_1}(t), W_{j_2}(t), W_{j_3}(t))$ is part of the triangulation. Notice the vector $V_i-V_{i\pm 1}\neq 0$ is not parallel to $T_{V_i}\Gamma$.}
\label{local}
\end{center}

Given $V=(x,y) \in \mathbb{C}^2$, let $V^\perp:=(-y, x)$. If $\langle , \rangle$ denotes the usual dot product, then the area of $T(t)$ is: \\
$$A(T(t))= (\lambda-1)\langle V_i-V_{i\pm 1}, \overrightarrow{k}_i^\perp\rangle t+O(t^2).$$

Recall $\lambda \neq 1$, and $\langle W, \overrightarrow{k_i}^\perp \rangle =0$ if and only if $W \in T_{V_i}\Gamma$, which is not the case for $V_i-V_{i\pm 1}$.\\

For the second part, recall that the vertices of each triangle $T_i(t)$ all lie on $\Gamma$, which is a conic. Therefore none of the angles of any $T_i$ tend to $\pi$, and hence the circumcenter $C_i(t)$ of each triangle in the limit is also finite.\\

Finally, for the third part we can even give precise coordinates of the centroid of a degenerate triangle $T_i=(V_i, V_i, V_{i\pm 1})$: $$CM_2(T_i)=\frac{2V_i +V_{i\pm 1}}{3}.$$
This is because the centroid of a non-degenerate triangle is obtained by the intersection of the medians. Recall the centroid divides every median in ratio $2:1$. This provides the formula given above.
\end{proof}


\section{Locus of the Circumcenter of Mass for Poncelet Polygons}

\begin{thm}
Let $\gamma, \Gamma$ be a pair of conics in general position that admit a $1$-parameter family of Poncelet $n$-gons $P_t$, with not-constant zero area. Then the locus $CCM(P_t)$ is also a conic.
\end{thm}

\begin{proof}
Recall that $\Gamma$ intersects the line at infinity in two different points, say $L$ and $M$. \\

Observe that the $x$ and $y$ coordinates of $CCM$ on ${\cal E}$ are meromorphic $D_n$-invariant functions, with $4n$ simple poles: $2n=|D_n|$ of them coming from each one of the flags that give rise to the polygon with a vertex in $L$, and $2n$ more coming from the polygon with one of the vertices in $M$.\\

Say $P$ is a Poncelet polygon with a vertex at infinity. Denote by $L_+, L_-$ and $M_+, M_-$ the adjacent vertices to $L$ and $M$ respectively. Then $CCM(P)$ coincides with the circumcenter of $(L_{-}, L, L_{+})$ (in case $L$ is a vertex of $P$) or the circumcenter of $(M_-, M, M_+)$ (if $M$ is a vertex of $P$), since this is the term that overpowers in formula (\ref{CCM}). The circumcenter of $(L_-,L,L_+)$ coincides the intersection of the perpendicular bisector of $L_{-}L_{+}$ at infinity, and similarly for the circumcenter of $(M_-, M, M_+)$.\\

Let $z$ be a local holomorphic parameter of ${\cal E}$ at $L$. Then

$$x(z)=\frac{u_1}{z}+u_3+\ldots , \hspace{1.5cm} y(z)=\frac{u_2}{z}+u_4+\ldots ,$$

where $[u_1: u_2:0]$ is the intersection of the perpendicular bisector of $L_-L_+$ at infinity.\\

Similarly, if $w$ is a local holomorphic parameter of ${\cal E}$ at $M$, then

$$x(w)=\frac{v_1}{w}+v_3+\ldots , \hspace{1.5cm} y(w)=\frac{v_2}{w}+v_4+\ldots ,$$

where $[v_1: v_2:0]$ is the intersection of the perpendicular bisector of $M_-M_+$ at infinity.\\

Consider $f(x,y)=Ax^2+By^2+Cxy+Dx+Ey$. We want to find $A,B,C,D,E \in \mathbb{C}$ such that $f$ has no poles at $L$ and $M$.\\

Expanding $f$ in the local parameters $z$ and $w$ for $L$ and $M$:

\begin{align*}
f(x(z), y(z)) =& \frac{1}{z^2}\left(u_1^2A + u_2^2B + u_1u_2C \right) + \\ 
               & \frac{1}{z}[2u_1u_3A + 2u_2u_4B + (u_1u_4+u_3u_2)C + u_1D + u_2E] + \ldots
\end{align*}

\begin{align*}
f(x(w), y(w)) =& \frac{1}{w^2}\left(v_1^2A + v_2^2B + v_1v_2C \right) +\\
               & \frac{1}{w}[2v_1v_3A + 2v_2v_4B + (v_1v_4+v_3v_2)C+ v_1D + v_2E] + \ldots
\end{align*}

So we are looking for $A, B, C, D, E$ non-trivial solution to the system of linear equations

\begin{equation}
\begin{cases}
u_1^2A + u_2^2B + u_1u_2C = 0.\\
2u_1u_3A + 2u_2u_4B + (u_1u_4+u_3u_2)C + u_1D + u_2E = 0.\\
v_1^2A + v_2^2B + v_1v_2C = 0.\\
2v_1v_3A + 2v_2v_4B + (v_1v_4+v_3v_2)C+ v_1D + v_2E = 0.
\end{cases}
\label{System}
\end{equation}

Recall $u_i, v_i$ are constants that depend on $\gamma$, $\Gamma$ and the relative position of $\gamma$ with respect to $\Gamma$.  \\

Since we have more variables than equations, the system is consistent and the solution set is at least $1$-dimensional.\\

If $A,B, C, D, E$ is a non-trivial solution of (\ref{System}), then $$f(x,y)=Ax^2+By^2+Cxy+Dx+Ey$$ has no poles on ${\cal E}$, meaning this function is constant, and so the coordinates $x(t),y(t)$ of $CCM(P_t)$ satisfy
$$Ax(t)^2+By(t)^2+Cx(t)y(t)+Dx(t)+Ey(t)=F,$$
for some $F \in \mathbb{C}$. Hence, the locus of $CCM(P_t)$ is a conic.\\
\end{proof}


\section{Area invariants of Poncelet polygons}

Through numerical experiments, Dan Reznik has found several invariants related to billiard trajectories. There is one of them with its proof:

\begin{thm}
Let $\gamma$, $\Gamma$ be two concentric ellipses in general position admitting a $1$-parameter family of Poncelet $n$-gons. Given $P=(V_1, \ldots, V_n)$ one of these Poncelet $n$-gons, let $Q$ be a new polygon formed by the tangent lines to $\Gamma$ at $V_i$ (see Figure \ref{invariant-areas}).
If $n$ is even, then $A(P)\cdot A(Q)$ stays constant within the Poncelet family.
\end{thm}

\begin{figure}
\centering
\includegraphics[scale=0.32]{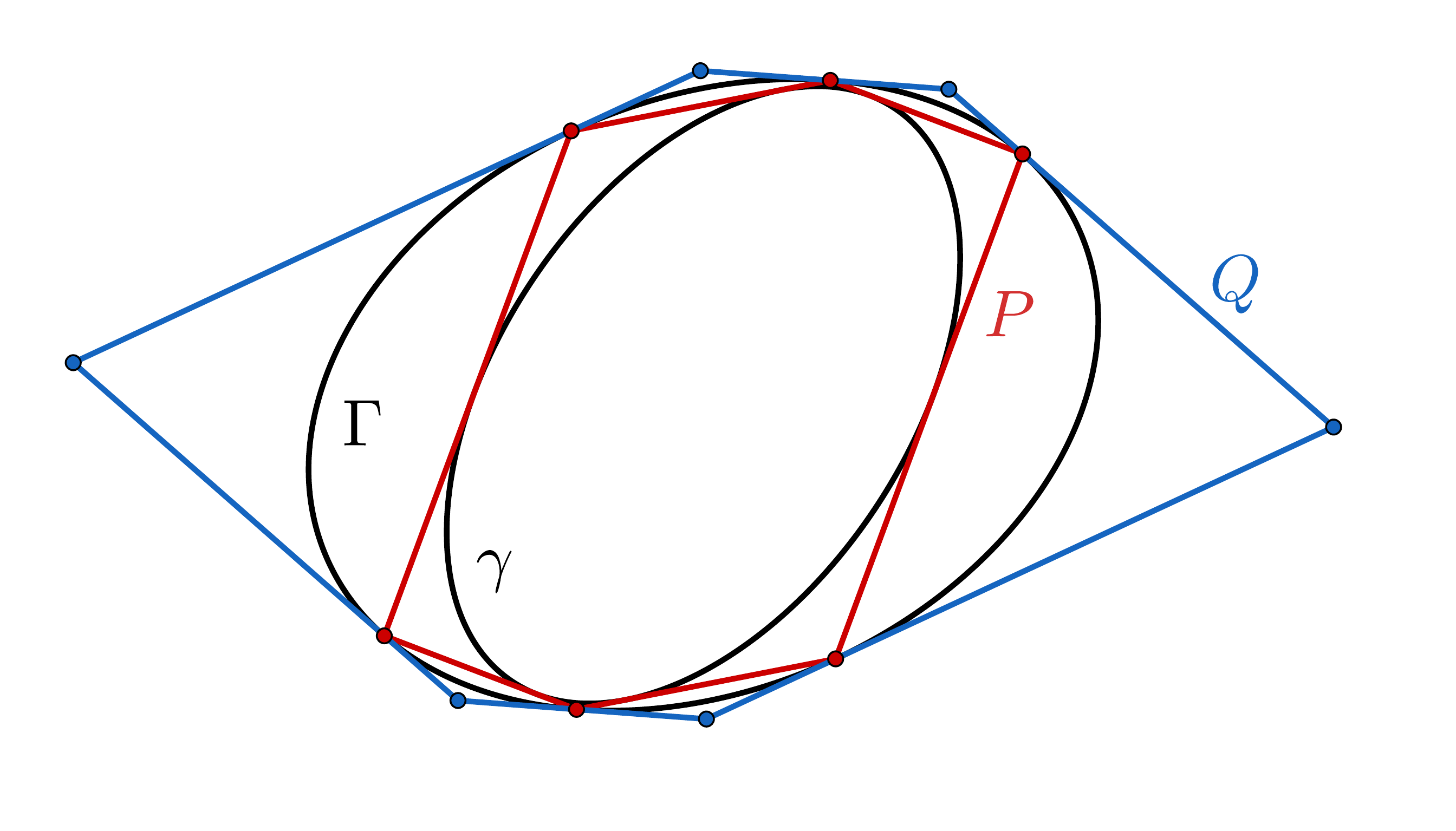}
\caption{The ellipses $\gamma$, $\Gamma$ admit in this case a $1$-parameter family of Poncelet hexagons. $P$ has vertices in $\Gamma$ and sides tangent to $\gamma$, while $Q$ has sides tangent to $\Gamma$.}
\label{invariant-areas}
\end{figure}

\begin{proof}
By applying an equiaffine transformation, we can assume $\gamma$ is a circle, with both $\gamma$ and $\Gamma$ centered at the origin. As before, the product of the areas can be thought of as a meromorphic function $g:{\cal E} \rightarrow \hat{\mathbb{C}}$. It would be enough to check that $g$ has no poles to conclude that $g$ must be constant.\\

Notice $g$ may have a pole if one of the polygons $P$ or $Q$ has a vertex at infinity. \\

If $Q$ had a vertex at infinity, this would imply that $T_{V_i}\Gamma$ and $T_{V_{i+1}}\Gamma$ are parallel. This happens only when $-V_{i}=V_{i+1}$. Since $n$ is even, and the Poncelet map commutes with the reflection in the origin, one has that $-V_i=V_{i+\frac{n}{2}}$. So, $V_{i+1}=V_{i+\frac{n}{2}}$: this is possible only when $n=4$, $V_1=V_4$, $V_2=V_3$. Notice that $V_i \in \gamma\cap \Gamma$  for all $i$, thus $P$ is a degenerate polygon, with finite vertices. In this case, $A$ has a simple pole at $Q$, but a simple zero at $P$. Then $g$ has no pole in this case.\\

If $Q$ has only finite vertices and $P$ has a point at infinity (say $V_n$), then the lines connecting $V_n$ with $V_1$ and $V_n$ with $V_{n-1}$ are parallel and by symmetry with respect to the reflection in the origin, $V_1=-V_{n-1}$. \\

By symmetry with respect to the reflection in the origin, one also has that $V_{i}=-V_{n-i}$, for all $i$, giving a pairing between finite vertices. Since $n$ is even, the only possibility is that $P$ is a degenerate polygon, with $V_{n/2}=V_n$ and
	\begin{itemize}
	\item $V_{n/4}, V_{3n/4} \in \gamma \cap \Gamma$ if $n\equiv 0 (\text{mod }4)$.
	\item $V_{\frac{n-2}{4}}=V_{\frac{n+2}{4}}, V_{\frac{3n+2}{4}}=V_{\frac{3n-2}{4}}$ and hence their tangent to $\Gamma$ is also tangent to $\gamma$, if $n\equiv 2 (\text{mod }4)$.
	\end{itemize}
In this case, $A$ has a simple pole at $P$, but because $P$ is degenerate, $Q$ is also degenerate (with finite vertices) and so, $A$ has a simple zero at $Q$. Hence $g$ has no pole at $P$ and therefore is constant.
\end{proof}


\section{Moving to other geometries}

The centers $CM_0$, $CM_2$ and $CCM$ still make sense in different geometries (for instance, \cite{CCM} goes over the definition and properties of $CCM$ in spherical and hyperbolic geometry). One may feel tempted to look into spherical and hyperbolic geometry and see if any of these theorems still hold. Unfortunately, the scenario doesn't look very promising.\\

For instance, consider $\mathbb{S}^2$. As described in \cite{Izmestiev-conics}, a spherical conic is the intersection of the sphere $\mathbb{S}^2$ with a quadratic cone. It doesn't make any harm to assume that all our configuration of spherical conics and polygons are contained in the northern hemisphere: $$\mathbb{S}^2_+:=\{(x,y,z): x^2+y^2+z^2=1, z>0 \}.$$
If $\pi: \mathbb{S}^2_+ \rightarrow \{(x, y, 1) : x,y \in \mathbb{R}\}\cong \mathbb{R}^2$  is the central projection, notice $\pi$ and its inverse $\pi^{-1}$ preserve conics and geodesics. Hence the Poncelet porism still holds in the spherical case.\\

Suppose $\gamma, \Gamma \subset \mathbb{S}^2_+$ are two conics that admit a $1$-parameter family of Poncelet triangles. For the specific case of triangles, it happens that the average of the vertices (projected into the sphere) coincides with the intersection of the medians. That is, if $T=(V_1, V_2,V_3)$ with $V_i \in \mathbb{S}^2$ is a spherical triangle, and $CM_2(T)\in \mathbb{S}^2$ is the intersection of the medians of $T$, then:

$$CM_2(T)=\frac{V_1+V_2+V_3}{||V_1+V_2+V_3||}=CM_0(T).$$ 

For this specific case of triangles, denote $CM(T):=CM_2(T)=CM_0(T)$.\\

\begin{figure}
\centering
\includegraphics[scale=0.25]{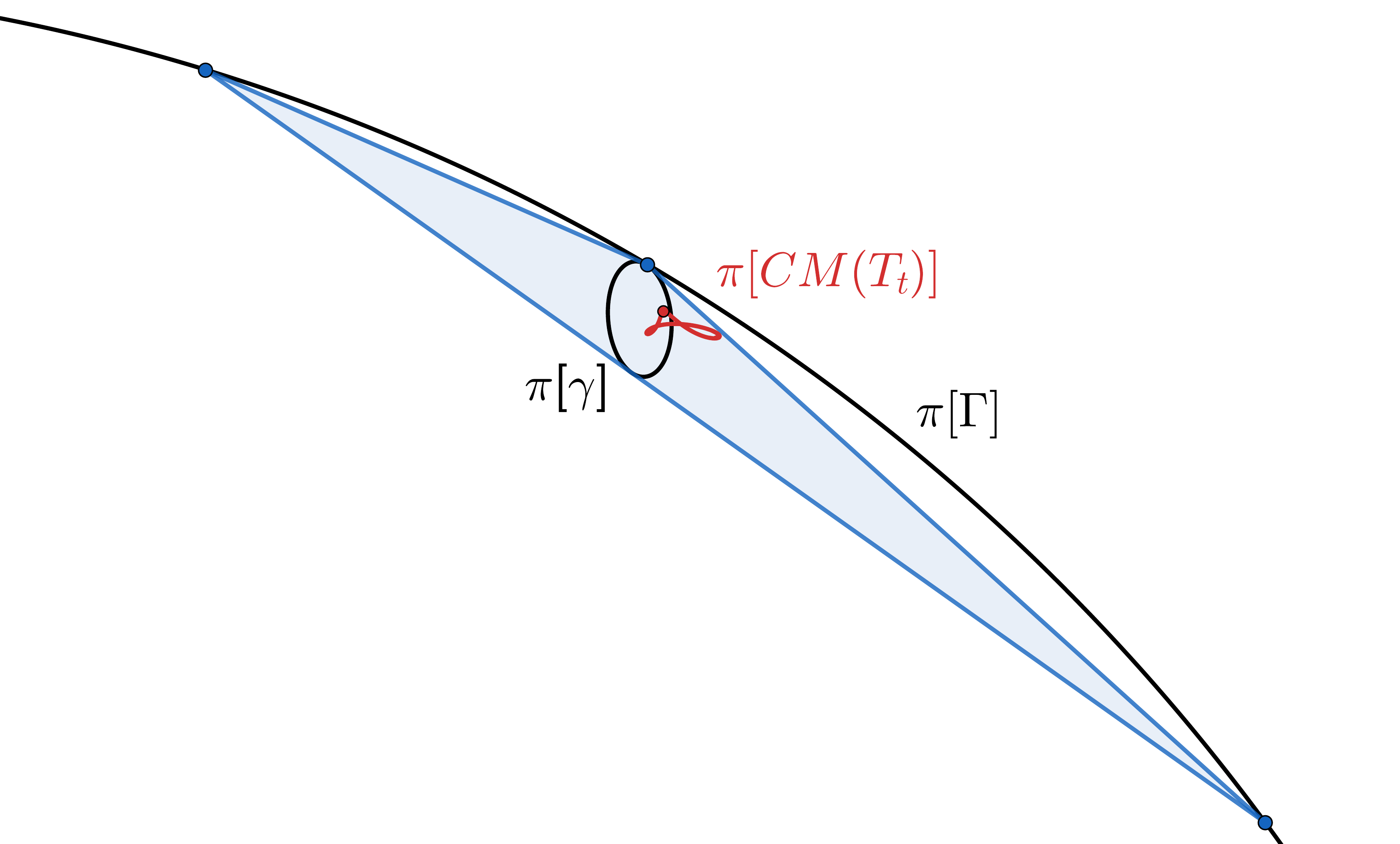}
\caption{The locus of the center of mass of Poncelet triangles is not a conic.}
\label{spherical}
\end{figure}

Figure \ref{spherical} shows the projection of a configuration of conics $\gamma, \Gamma \subset \mathbb{S}^2_+$ that admit a $1$-parameter family of Poncelet triangles $T_t \subset \mathbb{S}^2_+$. One can observe $\pi[CM(T_t)]$ is not a conic, so $CM(T_t)\subset \mathbb{S}^2$ is also not a conic.


\section{Acknowledgments} 

This paper wouldn't be possible without the guidance received from S. Tabachnikov, who suggested the topic in the first place. I'm deeply grateful for all the stimulating discussions and all the insightful comments and corrections.\\
Many thanks to A. Akopyan, D. Rudenko and M. Ghomi for the valuable suggestions and to R. Schwartz, whose work and ideas inspired this paper.\\
I also acknowledge the support received from the U.S. National Science Foundation through the grant DMS-1510055.

\bibliographystyle{abbrv}
\bibliography{CCM-Poncelet}

\end{document}